\documentclass[12pt, twoside]{article}
\usepackage{amsmath, amssymb, amsthm}
\usepackage{enumerate}
\usepackage{graphicx,subcaption}
\usepackage{array}

\usepackage{graphicx,color}

\theoremstyle{plain}
\newtheorem{theorem}{Theorem}[section]
\newtheorem{prop}{Proposition}[section]
\newtheorem{lemma}{Lemma}[section]

\theoremstyle{remark}
\newtheorem{remark}{Remark}[section]
\newtheorem{exmp}{Example}[section]

\def\oo{{\rm {o}}} 
\def\OO{{\rm {O}}}

\begin{document}

\baselineskip 6truemm

\bigskip
\centerline{\Large\bf Convergence rate for the}
\centerline{\Large\bf longest $T$-contaminated runs of heads.}
\centerline{\Large\bf Paper with detailed proofs}
\vskip0.3truecm

\centerline{
	\renewcommand{\thefootnote}{1}
	Istv\'an Fazekas,\footnotemark 
	\ Borb\'ala Fazekas,
	\ Michael Ochieng Suja
		\footnotetext{fazekas.istvan@inf.unideb.hu}
}
\vskip0.3truecm
\centerline{University of Debrecen, Hungary}
\vskip0.4truecm

\begin{abstract}
We study the length of $T$-contaminated runs of heads in the well-known coin tossing experiment.
A $T$-contaminated run of heads is a sequence of consecutive heads interrupted by $T$ tails.
For $T=1$ and $T=2$ we find the asymptotic distribution for the first hitting time 
of the $T$ contaminated run of heads having length $m$; furthermore, we obtain a limit theorem
for the length of the longest $T$-contaminated head run.
We prove that the rate of the approximation of our accompanying distribution for the length of 
the longest $T$-contaminated head run is considerably better than the previous ones.
For the proof we use a powerful lemma by Cs\'aki, F\"oldes and Koml\'os, see \cite{CFK}.
\end{abstract}
\section{Introduction}
\setcounter{equation}{0}
In this paper, we shall consider the well-known coin tossing experiment.
Let $p$ denote the probability of heads, so the probability of tails is $q=1-p$.
We toss the coin $N$ times.

We shall study the length of consecutive heads interrupted by $T$ tails.
A subsequence containing $T$ tails while all the other values are heads 
will be called a $T$-contaminated run of heads.

The most famous known results concern the length of the pure head runs.
The case of the fair coin was studied in the classical paper of Erd\H{o}s and R\'enyi \cite{ErReny75}.
Later several papers were published on this topic.
E.g. \cite{Novak} studied the accuracy of the approximation to the distribution of the length of the longest head run 
in a Markov chain (see also the references in \cite{Novak}).

An early paper obtaining almost sure limit results for the length of the longest runs 
containing at most $T$ tails is \cite{ErRe75}.
Later several papers were devoted to the topic.
F\"oldes, in \cite{Fol79}, presented asymptotic results for the distribution of the number of 
$T$-contaminated head runs, 
the first hitting time of a $T$-contaminated head run having a fixed length, 
and the length of the longest $T$-contaminated head run.
M\'ori in \cite{Mori93} obtained a so called almost sure limit theorem for the longest $T$-contaminated head run.

In the paper \cite{GSW86}, Gordon, Schilling, and Waterman applied extreme value theory to obtain the 
asymptotic behaviour of the  expectation and the variance of the length of the longest $T$-contaminated head run.
Also in \cite{GSW86}, accompanying distributions were obtained for the length of the longest $T$-contaminated 
head run. 
In \cite{FM}, it was shown that the accompanying distributions of \cite{GSW86} 
can be obtained by the method of F\"oldes \cite{Fol79}.

After minor algebraic manipulation, Theorem 1 of \cite{GSW86} states the following.
\begin{prop} \label{GSW} 
Let $\mu(N)$ denote the length of the longest $T$-contaminated run of heads 
during the coin tossing experiment of length $N$.
Let 
\begin{equation}  \label{m0}
m_0(N) = \log (qN) + T \log(\log(qN)) + T \log (q/p) - \log(T!),
\end{equation}
where $\log$ denotes the logarithm of base $1/p$.
Let $ [ m_0(N) ] $ denote the integer part of $m_0(N)$ and let 
$ \{ m_0(N) \} $ denote the fractional part of $m_0(N)$.
Then
$$
P \left(\mu(N)- [m_0(N)] <k \right) =
\exp \left( -p^{ k-  \{ m_0(N) \}} \right) + \rm{o}(1) .
$$
\end{prop}
However, numerical experiments show that the approximation offered by Proposition \ref{GSW} is quite weak.
Therefore the aim of this paper is to improve the above result for the most important cases of $T=1$ and $T=2$.

Our main result is Theorem \ref{propRun}.
Like Proposition \ref{GSW}, our theorem offers an accompanying distribution sequence for the distribution of the 
centralized version of $\mu(N)$.
We show that the rate of the approximation in our new theorem is $\OO\left(1/(\log(N))^2\right)$.
We shall see, that for $T=1$ and $T=2$ the rate of the approximation in Proposition \ref{GSW} is 
$\OO\left(\log(\log(N))/\log(N)\right)$, so our result considerably improves the former result.
We also obtain a result for the first hitting time of the $T$ contaminated run of heads having length $m$, see
Theorem  \ref{propT12km}.
In section \ref{SectSimu} we present some simulation results supporting that our new approximation 
excels the former ones.

The main tool of the proof is a powerful lemma by Cs\'aki, F\"oldes and Koml\'os, see \cite{CFK}.
We quote this lemma as Lemma \ref{MLsf}.
In order to apply this lemma, we need several elementary but quite long calculations, see lemmas 
\ref{lemT1kmSI}, \ref{lemT2kmSI}.
\section{The first hitting time and the longest run}
\setcounter{equation}{0}
Consider the well-known coin tossing experiment.
Let $p$ be the probability of heads and let $q=1-p$ be the probability of tails.
Here $p$ is a fixed number with $0 < p < 1$.
We toss a coin $N$ times independently.
As usual, we shall write $1$ for heads and $0$ for tails.
So we shall consider independent identically distributed random variables $X_1, X_2, \dots ,X_N$
with $P(X_i = 1) = p$ and $P(X_i = 0) = q$, $i=1, 2, \dots ,N$.

Let $T\ge 0$ be a fixed non-negative integer.
We shall study the $T$-interrupted runs of heads.
It means that there are $T$ zeros in an $m$ length sequence of ones and zeros.
So let $m$ be a positive integer.
Let denote $A_n=A_{n,m}$ the event that there are precisely $T$ zeros in the sequence 
$X_n, X_{n+1}, \dots, X_{n+m-1}$.
So $A_{n,m}$ is the event that there is a $T$-interrupted run of ones at positions $n, {n+1}, \dots, {n+m-1}$.

Let $\tau_m$ be the first hitting time of the $T$-contaminated run of heads having length $m$.
We shall find the asymptotic distribution of $\tau_m$ as $m\to \infty $ for $T=1$ and for $T=2$.
\begin{theorem}  \label{propT12km}
	Let $T=1$ or $T=2$, $0<p<1$.
	Let $\tau_m$ be the first hitting time of the $T$ contaminated run of heads having length $m$.
	Then, for $x> 0$, 
	$$
	P( \tau_m  \alpha P(A_1) > x) \sim e^{-x}
	$$
	as $m\to \infty$.
	Here if $T=1$, then $\alpha= q + \frac{2 p^{m-1} -1}{m}$ and $P(A_1) = m  p^{m-1} q$.
	When $T=2$, then $\alpha= q - \frac{2}{m}$, $ P(A_1) = {m\choose 2}p^{m-2} q^2$.
\end{theorem}
\begin{remark}\label{remT2}
One can show, that Theorem \ref{propT12km} is valid for $T=2$ 
	with $\alpha= q - \frac{2}{m} + \frac{2(m-2)}{m} p^{m-2} - \frac{2(m-4)}{m} p^{m-1} $, too (see Remark \ref{remT2kmSI}).
\end{remark}

Now, we turn to the length of the longest contaminated run of heads.
\begin{theorem}  \label{propRun}
	Let $T=1$ or $T=2$, and let $0<p<1$ be fixed.
	Let $\mu(N)$ be the length of the longest $T$ contaminated run of heads during tossing $N$ times a coin.
Let 
\begin{eqnarray}  \label{mN}
m(N) &=& \log(qN) + T\log(\log(qN)) + \\
\nonumber
&+&  T^2 \frac{\log(\log(qN))}{c \log(qN)} -  \frac{T}{c q_0 \log(qN)} 
     - \frac{T^3}{2 c} \left( \frac{\log(\log(qN) )}{\log(qN)} \right)^2  + \\
\nonumber
&+& T^2 \frac{\log(\log(qN))}{cq_0 (\log(qN))^2} + T^3 \frac{\log(\log(qN))}{(c \log(qN))^2} + \\
\nonumber
&+& \left(T \log\left(\frac{q}{p}\right) -\log (T!)\right) 
\left(1 + \frac{T}{c \log(qN)} - T^2 \frac{\log(\log(qN))}{c (\log(qN))^2}  \right) , 
\end{eqnarray} 
where $\log $ denotes the logarithm of base $1/p$,
$c = \ln(1/p)$, $\ln$ denotes the natural logarithm of base $e$, and   $q_0= \frac{2q}{2+Tq-q}$.
Let $[m(N)]$ denote the integer part of $m(N)$, while $\{m(N)\}$ denotes the fractional part of $m(N)$,
i.e. $\{m(N)\} = m(N)- [m(N)]$.

Then 
\begin{eqnarray} \label{mulim}
 && P(\mu(N) - [m(N)] < k ) =  \\
&& \qquad \qquad  = e^{\displaystyle{- p^{(k- \{m(N)\})
\left( 1-\frac{T}{c \log(qN)} +T^2 \frac{\log(\log(qN))}{c (\log(qN))^2} \right)}}} 
\left (1+ \OO\left( \frac{1}{(\log N)^2} \right) \right) \nonumber
\end{eqnarray}
for any integer $k$, 
where $f(N)=\OO(h(N))$ means that $f(N)/h(N)$ is bounded as $N \to \infty$.	
\end{theorem}
\begin{remark}
Using our method, for $T=1$ and $T=2$ and for $m_0(N)$ from \eqref{m0},
we obtain that the rate of convergence in Proposition \ref{GSW} is $\OO\left(\log(\log(N))/\log(N)\right)$, that is 
$$
P \left(\mu(N)- [m_0(N)] <k \right) =
\exp \left( -p^{ k-  \{ m_0(N) \}} \right) \left(1+ \OO\left(\log(\log(N))/\log(N)\right) \right) .
$$
So our Theorem  considerably improves Theorem 1 of \cite{GSW86} in the cases of $T=1$ and $T=2$.
\end{remark}
\section{Preliminary Lemmas}
\setcounter{equation}{0}
The following lemma of \cite{CFK} will play a fundamental role in our proofs.
\begin{lemma}	\label{MLsf}
	(Main lemma, stationary case, finite form of Cs\'aki, F\"oldes, Koml\'os \cite{CFK}.)
	Let $X_1,X_2,\dots$ be any sequence of independent random variables, 
and let $\mathcal{F}_{n,m}$ be the $\sigma$-algebra generated by the random variables $X_n,X_{n+1},\dots,X_{n+m-1}$. 
Let $m$ be fixed and let $A_n=A_{n,m}\in\mathcal{F}_{n,m}$. 
Assume that the sequence of events $A_n = A_{n,m}$ is stationary, that is
$P(A_{i_1+d} A_{i_2+d} \cdots A_{i_k+d})$ is independent of $d$.
	 
	Assume that there is a fixed  number $\alpha$, $0<\alpha \le 1$, 
	such that the following three conditions hold for some fixed $k$ with $2 \le k \le m$, 
	and fixed $\varepsilon$ with $0< \varepsilon < \min\{ p/10, 1/42\}$ 
	\begin{enumerate}[(SI)]
		\item 
		\begin{equation*}
			|P(\bar{A_2} \cdots \bar{A_k}| A_1)-\alpha| <\varepsilon ,
		\end{equation*} 
		\item
		\begin{equation*}
			\sum_{k+1 \le i \le 2m} P(A_i|A_1) <\varepsilon ,
		\end{equation*}
		\item
		\begin{equation*}
			P(A_1) < \varepsilon/m .
		\end{equation*}
	\end{enumerate}
	Then, for all $N > 1$,
	\begin{equation*}
		\left|\frac{P(\bar{A_2} \cdots \bar{A}_N|A_1)}{P(\bar{A_2} \cdots \bar{A}_N)} - \alpha \right| <7 \varepsilon
	\end{equation*}
	and 
	\begin{equation} \label{MLsfe}
		e^{-(\alpha+10\varepsilon)NP(A_1)-2mP(A_1)} < P(\bar{A_1} \cdots \bar{A}_N) <e^{-(\alpha-10\varepsilon)NP(A_1)+2mP(A_1)}.
	\end{equation}
\end{lemma}
First we present some preliminary results which can have their own interest.
We shall check conditions (SI)-(SIII) of Lemma \ref{MLsf} if $k=m$.
Next remarks show that conditions (SII) and (SIII) are true for any $T$ if $m$ is large enough.

\begin{remark}  \label{RemSIII}
	Consider condition (SIII).
	We show that (SIII) is true for any $T$ if $m$ is large enough.
	We have
	\begin{equation*}
		P(A_1) = {m \choose T} p^{m-T} q^T \le \frac{m^T}{T!} p^{m-T} q^T < \frac{\varepsilon}{m},
	\end{equation*}
	if
	\begin{equation}  \label{SIII}
		m^{T+1}  p^m   < \varepsilon \left(\frac{p}{q}\right)^T T!,
	\end{equation}
	and the last inequality is satisfied for any fixed positive $\varepsilon$ if $m$ is large enough.
	So we always can assume (SIII) of Lemma \ref{MLsf} for any $T$.
\end{remark}

\begin{remark}  \label{RemSII}
	Consider condition (SII).
	\begin{equation*}
		P(A_i|A_1)=P(A_i) = {m \choose T} p^{m-T} q^T \le \frac{m^T}{T!} p^{m-T} q^T ,
	\end{equation*}
	if $i>m$ because of independence.
	So
	\begin{equation*}
		\sum_{i=m+1}^{2m}  P(A_i|A_1)= m P(A_1) = m {m \choose T} p^{m-T} q^T \le  m\frac{m^T}{T!} p^{m-T} q^T < \varepsilon,
	\end{equation*}
	therefore we obtain again condition \eqref{SIII}.
	So condition (SII) is true if $m$ is large enough.
\end{remark}

To check condition (SI), first we fix $T=1$.

\begin{lemma}  \label{lemT1kmSI}
	Condition {\rm{(SI)}} of Lemma \ref{MLsf} is satisfied for $T=1$ and $k=m$ in the following form
	\begin{equation}  \label{condSI}
		|P(\bar{A_2} \cdots \bar{A_k}| A_1)-\alpha| <\varepsilon ,
	\end{equation} 
	with $\alpha= q + \frac{2 p^{m-1} -1}{m}$.
\end{lemma}
\begin{proof} Fix $T=1$ and $k=m$.
	To calculate the probability of the event $A_1 \bar{A_2} \cdots \bar{A_k}$, we divide it into parts.
	
	Consider those $0-1$ sequences $X_1, X_2, \dots, X_{2m-1}$ which belong to $A_1 \bar{A_2} \cdots \bar{A_k}$.
	If the first member of the sequence is $0$, then the members on the places $m+1,\dots, 2m-1$ should be ones.
	So this part of $A_1 \bar{A_2} \cdots \bar{A_k}$ has probability $q p^{m-1} p^{m-1} $.
	
	If the first member is $1$, then  $X_{m+1}$ should be zero.
	If we fix that the only zero in $X_1, X_2, \dots, X_{m}$ is at the $l$th place with $2\le l \le m-1$, 
	then besides $X_{m+1}=0$
	there should be at least one zero among $X_{m+2}, \dots, X_{m+l}$.
	Its probability is $q p^{m-1} q (1-p^{l-1}) $.
	Finally, if the only zero in $X_1, X_2, \dots, X_{m}$ is at the $m$th place, then we need only $X_{m+1}=0$.
	Its probability is $q p^{m-1} q $. 
	Therefore
	\begin{equation*}
		P( A_1 \bar{A_2} \cdots \bar{A_k}) = q p^{m-1} \left( p^{m-1} +q \sum_{i=1}^{m-2} (1-p^i) +q\right)
	\end{equation*}
	and 
	\begin{equation*}
		P( \bar{A_2} \cdots \bar{A_k} | A_1) = q + \frac{2 p^{m-1} -1}{m}.
	\end{equation*}
\end{proof}		

\begin{remark}
	We shall use the following known formulae
	\begin{equation*}
		\sum_{i=a}^{b}ix^{i-1} = \frac{bx^{b+1} -(b+1)x^b -(a-1)x^a + ax^{a-1}}{(x-1)^2}
	\end{equation*}
and
\begin{align*}
	\sum_{i=a}^{b}i(i-1)x^{i-2} & = \frac{b(b+1)x^b -a(a-1)x^{a-1} -b(b+1)x^{b-1} +a(a-1)x^{a-2}}{(x-1)^2} -\\
	&-\frac{2\left[bx^{b+1} -(a-1)x^a -(b+1)x^b +ax^{a-1}\right]}{(x-1)^3} .
\end{align*}
\end{remark}

\begin{lemma}\label{lemT2kmSI}
	Condition {\rm{(SI)}} of Lemma \ref{MLsf} is satisfied for $T=2$ and $k=m$ in the following form
	
	\begin{equation}  \label{condSI}
		|P(\bar{A_2}\bar{A_3} \cdots \bar{A}_m| A_1)-\alpha| <\varepsilon ,
	\end{equation} 
	with $\alpha= q - \frac{2}{m} + \OO(p^m)$
	as $m\to \infty$.
\end{lemma}

\begin{proof}
	Fix $T=2$ and let $q= 1-p$. We write 1 for heads and 0 for tails.
	\begin{equation*}
	P(\bar{A_2}\bar{A_3} \cdots \bar{A}_m| A_1) = \frac{P(A_1\bar{A_2} \cdots \bar{A}_m)}{P(A_1)} .
	\end{equation*}
Here
$ P(A_1) = {m\choose 2}p^{m-2} q^2$.

To calculate $P(A_1\bar{A_2} \cdots \bar{A}_m)$, we divide the event $A_1\bar{A_2} \cdots \bar{A}_m$ into parts.

I.	If the first element is 0, then the $(m+1)$\textsuperscript{st} should be 1.\\
	$\underbrace{0, 1, \cdots ,1,\overset{(k)}{0}, 1, \cdots}_{m} \underbrace{1,1,\cdots,1}_{k-1}, \cdots$ \\
	
	So the probability of this part is
	\begin{align*}
	{} &\sum_{k=2}^{m-1}q^2p^{m-2} \left(p^{m-1} + p^{m-2}q(m-k)\right) + q^2 p^{m-2}. p^{m-1}  \\
			= & (m-1) q^2 p^{2m-3} + \frac{(m-1)(m-2)}{2}q^3 p^{2m-4} .
	\end{align*}
This term is 'small', i.e it is of order $\OO(p^m)\cdot P(A_1)$.

II. Now let us turn to the case, when the first element is 1. 
Then the $(m+1)$\textsuperscript{st} element should be 0. 
Let the $k$\textsuperscript{th} and the $l$\textsuperscript{th} elements be zeros, $1 < k < l \le m$.\\

$\underbrace{1,\cdots,1,\overset{(k)}{0},1, \cdots,1,\overset{(l)}{0},\cdots}_{m},\underbrace{\underbrace{\underbrace{0,\cdots, 0, \cdots}_{k},\cdots}_{l}, \cdots}_{m-1}$\\

Then on places $m+2, \cdots, m+k$, there should be at least one 0 and on places $m+2, \cdots, m+l$ , 
there should be at least two zeros.

%
II/1. If $k=2$, then on the places $m+1, m+2$ should stay zeros. 
	Moreover, when $l\neq m$, there should be at least one 0 at places $m+3, \cdots, m+l$.
	
	$\underbrace{1, \overset{(2)}{0},1, \cdots,1,\overset{(l)}{0},\cdots}_{m}\underbrace{\underbrace{0,0,\cdots}_{l}
	\cdots}_{m}$
	
	 So the probability of this part is
	\begin{equation*}
		P(A)= \sum_{l=3}^{m-1} p^{m-2}q^2\cdot q^2(1-p^{l-2}) + p^{m-2}q^2q^2.
	\end{equation*}

II/2. 	Now, let us consider the case $k > 2$.
	We shall study separately the case $l < m$ and the case of $l = m$. 
	The first case is divided into two parts: $B$ and $C$, say.

II/211.	Let $k > 2$, $l < m$ and on the places $m+2, \cdots, m+k$, there are at least two 0's.\\
	$\underbrace{1,\cdots1,\overset{(k)}{0},\cdots,\overset{(l)}{0}, \cdots,1}_{m},\underbrace{0, \cdots,0,\cdots,0, \cdots}_{k}\cdots$ 
	
	The probability of this part is:\\
	\begin{equation*}
		P(B)= \sum_{k=3}^{m-2} p^{m-2}q^2 (m-k-1) \cdot q(1-p^{k-1}-(k-1) \cdot p^{k-2}q) .
	\end{equation*}

II/212.	Let $k > 2$, $l < m$ and on the places $m+2, \cdots, m+k$, there is precisely one 0. 
	Moreover, on the places $m+k +1, \cdots, m+l$ there is at least one 0. \\
	$\underbrace{1, \cdots,1,\overset{(k)}{0},\cdots,\cdots, \overset{(l)}{0}, \cdots,1}_{m},\underbrace{\underbrace{0,  \cdots,0, \cdots}_{k},\cdots,\cdots,0,\cdots}_{l}\cdots$
	
	The probability of this part is:\\
	\begin{equation*}
		P(C) = \sum_{k=3}^{m-2} p^{m-2}q^2 \sum_{l= k+1}^{m-1}q (k-1) q p^{k-2}(1-p^{l-k}) .
	\end{equation*}
	 

II/22.	Let $k > 2$ and $l = m$. Then on the places $m+2, \cdots, m+k$, there should be at least one 0.\\
	$\underbrace{1, \cdots\cdots, \overset{(k)}{0},\cdots\cdots, \overset{(l)}{0}}_{m},\underbrace{0,\cdots\cdots,0}_{k}$\\
	
	The probability of this part is:	
	\begin{equation*}
		P(D) = \sum_{k=3}^{m-1} q^2 p^{m-2}q (1- p^{k-1}) .
	\end{equation*}	

Now, we reshape the above expression of $P(C)$.
	\begin{align*}
		P(C)={} & \sum_{k=3}^{m-2}p^{m-2} q^4 \Big[(m-k-1)(k-1)p^{k-2}-(k-1)\underbrace{\sum_{l=k+1}^{m-1} p^{l-2}}_{\frac{p^{k-1}-p^{m-2}}{q}}\Big]\\
		= &  p^{m-2}q^4 \left[\sum_{k=3}^{m-2}(m-3)(k-1) p^{k-2}- p \sum_{k=3}^{m-2}(k-2)(k-1) p^{k-3} -\right. \\		
		& \left.- \frac{p}{q} \sum_{k=3}^{m-2}(k-1) p^{k-2} + \sum_{k=3}^{m-2}(k-1) \frac{p^{m-2}}{q}\right]\\
		= & p^{m-2}q^4 \left[C_1 + C_2 + C_3 + C_4\right] .
	\end{align*}
Here
\begin{align*}
	C_1 = {} & (m-3) \sum_{i=2}^{m-3} i p^{i-1}\\
	= & (m-3) \frac{(m-3) p^{m-2} - (m-2) p^{m-3} - p^2 + 2p}{q^2} ,
\end{align*}
\begin{align*}
	C_2 = {} & -p\sum_{j= 2}^{m-3}(j-1) j p^{j-2}\\
	= & -p \left[\frac{(m-3)(m-2) p^{m-3} - 2p -(m-3)(m-2) p^{m-4} +2}{q^2} - \right. \\
	&- \left. 2 \frac{(m-3) p^{m-2} - p^2 - (m-2) p^{m-3} + 2p}{-q^3}\right] ,
\end{align*}
\begin{align*}
	C_3 = {} & -\frac{p}{q} \sum_{k=2}^{m-3} k p^{k-1}\\
	= & - \frac{p}{q}\frac{(m-3)p^{m-2} - (m-2) p^{m-3} - p^2 + 2 p}{q^2} ,
\end{align*}
\begin{align*}
	C_4 = {} \frac{p^{m-2}}{q} \frac{(m-1)(m-4)}{2} .
\end{align*}

From these, and omitting the 'small' (i.e. $\OO(p^m)$) terms, we obtain
\begin{equation*}
	\frac{P(C)}{P(A_1)} = \frac{p^{m-2}q^4 \left(C_1 + C_2 +C_3 + C_4\right)}{{m \choose 2} p^{m-2} q^2}\\
	\sim \frac{1}{{m \choose 2}} \left[p(1+q)(m-3) - \frac{p}{q}(3-q^2)\right] .
\end{equation*}
Then,
\begin{align*}
	P(A) = {} & (m-2) p^{m-2}q^4 - p^{m-2} q^4\left(p+p^2+ \cdots + p^{m-3}\right)\\
	= & (m-2) p^{m-2}q^4 - p^{m-2}q^4 p \frac{p^{m-3} -1}{p-1} .
\end{align*}
Therefore,
\begin{equation*}
	\frac{P(A)}{P(A_1)} \sim \frac{(m-2)q^2 -pq}{{m \choose 2}} .
\end{equation*}

Then 
$$
P(D) = (m-3)q^3 p^{m-2} - q^3 p^{m-2} p^2 \frac{p^{m-3}-1}{p-1}.
$$
Therefore
\begin{equation*}
	\frac{P(D)}{P(A_1)}\sim \frac{(m-3)q - p^2}{{m \choose 2}} .
\end{equation*}

Now we turn to $B$.
\begin{align*}
	P(B) = {}& \sum_{k=3}^{m-2} p^{m-2}q^3 m \left[1-p^{k-1} - (k-1) p^{k-2} q\right] \\
	- & p^{m-2}q^3 \sum_{k=3}^{m-2}(k+1)\left[1-p^{k-1}-(k-1)p^{k-2}q\right]\\
	= & V_1 + V_2 + V_3 + V_4 + V_5 + V_6,
\end{align*}
where
\begin{equation*}
	V_1 = mp^{m-2}q^3(m-4),
\end{equation*}
\begin{equation*}
	V_2 =-mp^{m-2}q^3
\left(p^2+p^3+ \cdots + p^{m-3}\right)=
 mp^{m-2}q^3 p^2 \left(\frac{p^{m-4}}{q}- \frac{1}{q}\right),
\end{equation*}
 \begin{equation*}
 V_3 = -mp^{m-2}q^4 \sum_{k=2}^{m-3}k p^{k-1} = -mp^{m-2}q^4 \left[\frac{(m-3)p^{m-2}-(m-2)p^{m-3}-p^2 + 2p}{q^2}\right],
\end{equation*}
 \begin{equation*}
 V_4= -p^{m-2}q^3 \sum_{k=3}^{m-2}(k+1) = -p^{m-2}q^3 \frac{(m+3)(m-4)}{2},
\end{equation*}
\begin{equation*}
	V_5 = p^{m-2}q^3\frac{1}{p}\sum_{k=4}^{m-1}k p^{k-1} = p^{m-3}q^3 \frac{(m-1)p^m - mp^{m-1}- 3 p^4 + 4 p^3}{q^2} ,
\end{equation*}
\begin{align*}
	V_6 = {} & p^{m-2}q^4 \sum_{k=3}^{m-2}k(k-1)p^{k-2} + p^{m-2}q^4 \sum_{k=2}^{m-3}kp^{k-1} \\
	= & p^{m-2}q^4\left[\frac{(m-2)(m-1)p^{m-2} -3\cdot 2 \cdot p^2 -(m-2)(m-1)p^{m-3} + 
	3\cdot 2\cdot p}{q^2} + \right. \\
	 & \left.  + \frac{ [ 2(m-2)p^{m-1} -2p^3 -(m-1)p^{m-2} +3p^2 ]}{q^3} \right.\\
	 & \left. + \frac{(m-3)p^{m-2} -(m-2)p^{m-3} -p^2+ 2p}{q^2}\right] .
  \end{align*}

From here
\begin{align*}
	\frac{P(B)}{P(A_1)} \sim {}&\frac{1}{{m \choose 2}}\left[m(m-4)q - mp^2- m(-p^2 +2p) -q\frac{(m+3)(m-4)}{2} +\right. \\
	& \left. + \frac{1}{qp}\left(4p^3 -3p^4\right) -3\cdot 2 \cdot p^2 + 3 \cdot 2\cdot p - \frac{4p^3}{q} 
	+ 6\frac{p^2}{q} -p^2 +2p \right] \\
	=& \frac{1}{{m \choose 2}}\left[q\frac{(m-4)(m-3)}{2} - 2mp +8p -7p^2 -\frac{7p^3}{q} + 10\frac{p^2}{q}\right] .
\end{align*}

Therefore
\begin{align*}
	P(\bar{A_2}\bar{A_3} \cdots \bar{A}_m &| A_1) \sim {} \frac{P(A) + P(B)+ P(C)+ P(D)}{P(A_1)}\\
=& \frac{1}{{m\choose2}}\left[\left((m-2)q^2 -pq\right) + \left(q\frac{(m-4)(m-3)}{2} - 2mp + 8p - 7p^2 -\right.\right. \\
	& -\left.\left. \frac{7p^3}{q} + 10\frac{p^2}{q}\right) + \left((m-3)p(1+q) - \frac{p}{q}(3-q^2)\right) + 
	\left((m-3)q -p^2\right)\right]\\
	= & q-\frac{2}{m} .
\end{align*}
\end{proof}
\begin{remark}\label{remT2kmSI}
A more careful calculation shows that Lemma \ref{lemT2kmSI} is valid for $T=2$ 
	with $\alpha= q - \frac{2}{m} + \frac{2(m-2)}{m} p^{m-2} - \frac{2(m-4)}{m} p^{m-1}  $, too.
\end{remark}
\section{Proofs of the main results}
\setcounter{equation}{0}
\begin{proof}[Proof of Theorem \ref{propT12km}.]
	We can apply Lemma \ref{MLsf} because its conditions are satisfied due to Remark \ref{RemSIII}, Remark \ref{RemSII},
	 and Lemma \ref{lemT1kmSI}, Lemma \ref{lemT2kmSI}.
	$$
	P( \tau_m  \alpha P(A_1) > x) = P( \tau_m  > N),
	$$
	where $N$ is the integer part of $\frac{x}{\alpha P(A_1)}$.
	In Lemma \ref{MLsf}, we can choose  $10\varepsilon=\varepsilon_0/m$, 
	where $\varepsilon_0$ is a fixed positive number.
	Let  $N_1= N-m+1$.
	So, by \eqref{MLsfe},
	\begin{equation} \label{MLsfe1}
		P( \tau_m  > N)=	P(\bar{A_1} \cdots \bar{A}_{N_1}) \sim 
		e^{-(\alpha \pm 10\varepsilon)N_1 P(A_1)\pm 2mP(A_1)} \sim
	\end{equation}
	\begin{equation*}
		\sim e^{-(\alpha \pm \frac{\varepsilon_0}{m}) \left(\frac{x}{\alpha P(A_1)}-m+1\right) P(A_1) }
		 e^{\pm 2mP(A_1)}  \sim
	\end{equation*}	
\begin{equation*}
		\sim e^{-(\alpha \pm \frac{\varepsilon_0}{m}) \frac{x}{\alpha}}
		e^{-(\alpha \pm \frac{\varepsilon_0}{m}) (-m+1) P(A_1)} e^{ \pm 2mP(A_1)} \sim e^{-x}
	\end{equation*}		
	as $m\to\infty$.
\end{proof}

\begin{proof}[Proof of Theorem \ref{propRun}.]
We shall give the proof for more general setting.
Assume that Lemma \ref{lemT1kmSI} and Lemma \ref{lemT2kmSI} are true for any $T$.
More precisely, we assume that condition {\rm{(SI)}} of of Lemma \ref{MLsf} is satisfied for any positive integer $T$
 and for $k=m$ in the following form	
	\begin{equation} \label{assump}
		|P(\bar{A_2}\bar{A_3} \cdots \bar{A}_m| A_1)-\alpha| <\varepsilon ,
	\end{equation} 
	with $\alpha= q - \frac{T}{m}$,
	where $\varepsilon=  \OO(p^m)$ as $m\to \infty$.

The above assumption will imply Theorem \ref{propRun} for any positive integer $T$.

Now, assumption \eqref{assump} and Remarks  \ref{RemSIII} and \ref{RemSII} imply, that 
Lemma \ref{MLsf} is satisfied with $\alpha = q- T/m$  and 
$$
\varepsilon= C m^{T+1} p^m.
$$
So we shall apply equation \eqref{MLsfe} of Lemma \ref{MLsf}, i.e.
	\begin{equation} \label{MLsfe+}
		e^{-(\alpha+10\varepsilon)NP(A_1)-2mP(A_1)} < P(\bar{A_1} \cdots \bar{A}_N) 
		<e^{-(\alpha-10\varepsilon)NP(A_1)+2mP(A_1)}
	\end{equation}
with the values of $P(A_1) = {m \choose T} p^{m-T} q^T $, $\alpha=q - \frac{T}{m}$ and $\varepsilon= C m^{T+1} p^m$.
We shall apply the above inequality for $m= [m(N)] +k$, where $m(N)$ is from equation \eqref{mN}.
For this $m$, direct calculations show that
\begin{equation}  \label{K12}
0< K_1 \le N m^T p^m \le K_2  <\infty.
\end{equation}
By inequality  \eqref{MLsfe+}, using $N_1= N-m+1$ instead of $N$, we have
\begin{eqnarray}  \label{e0}
P(\mu(N)  < m ) &=& P(\bar{A_1} \cdots \bar{A}_{N_1}) \sim e^{-(\alpha \pm 10\varepsilon)N_1 P(A_1) \mp 2mP(A_1)} = 
\nonumber \\
&=& e^{-(q- T/m) N_1 P(A_1)} e^{-(\pm 10\varepsilon)N_1 P(A_1)} e^{\mp 2mP(A_1)}. 
\end{eqnarray}
We shall show that the second and the third terms in \eqref{e0} converge to $0$, 
so the signs $\pm$ will not affect the result.

As $P(A_1) = {m \choose T} p^{m-T} q^T $, and the magnitude of $m$ is $\log N$, 
so the magnitude of the exponent of the third term in \eqref{e0} is
$$
(\log N)^{T+1} p^{\log N} = (\log N)^{T+1} /N,
$$
which converges to $0$ as $N\to\infty$.
So we can use the approximation $e^x \le 1+ C x$ for small values of $|x|$,
therefore we obtain
\begin{equation}  \label{e3}
e^{\pm 2mP(A_1)} = 1+ \OO\left((\log N)^{T+1} /N \right).
\end{equation}
Similarly, for the second term in \eqref{e0}, we have
\begin{equation}  \label{e2}
e^{-(\pm 10\varepsilon)N_1 P(A_1)} \sim e^{\pm C m^{T+1} p^m N_1 m^T p^m} \sim
e^{\pm C (\log N)^{2T+1} /N} = 1+ \OO\left((\log N)^{2T+1} /N \right).
\end{equation}
%

So, from formulae \eqref{e0}-\eqref{e2}, we obtain that 
\begin{equation} \label{mu1}
P(\mu(N)  < m ) =
e^{-\left(q- \frac{T}{m}\right) N_1 P(A_1) }
\left( 1+ \OO\left(1/(\log N)^2 \right) \right) .
\end{equation}  
As
$$
P(\mu(N) - [m(N)] < k ) = P(\mu(N) < m )
$$
with $m= [m(N)] + k = m(N) +k - \{ m(N) \}$, so we apply \eqref{mu1} for this form of $m$.

Now, the logarithm of the exponent in \eqref{mu1} is
\begin{eqnarray} \nonumber
L &=&\log \left( \left(q- \frac{T}{m}\right) N_1 P(A_1)  \right) =  \\
&=& \log\left(q- \frac{T}{m}\right) +\log N  + \log \left(m(m-1) \cdots (m-T+1) \right) +\log\left(p^m\right) +  \nonumber \\
&& \qquad +  \log((q/p)^T) - \log(T!) + \OO\left(\frac{\log N}{N }\right)=  \nonumber \\
&=& \log\left(q- \frac{T}{m}\right) +\log N  + \log \left(m^T- \frac{T(T-1)}{2} m^{T-1} +\OO(m^{T-2}) \right) -m +
 \nonumber \\
&& \qquad + \log((q/p)^T) - \log(T!) + \OO\left(\frac{\log N}{N }\right) =  \nonumber \\
&=& \log \left(q- \frac{T}{m}\right) +\log N  + \log (m^T) - \frac{\frac{T(T-1)}{2} m^{T-1}}{cm^T} 
+\OO\left(\frac{1}{m^2}\right) -m +  \nonumber \\
&& \qquad +\log((q/p)^T) - \log(T!) + \OO\left(\frac{\log N}{N }\right) =  \nonumber \\
&=& \log q  -\frac{T}{c q m} + \log N + T \log m  - \frac{ \frac{T(T-1)}{2} }{cm}-m + \nonumber  \\
     && \qquad +\log((q/p)^T) - \log(T!)+\OO\left(\frac{1}{(\log N)^2 }\right)= \nonumber \\     
&=& \log(qN) -\frac{T}{c q_0 m} + T \log m -m +\log((q/p)^T) - \log(T!) +\OO\left(\frac{1}{(\log N)^2 }\right) , \nonumber
\end{eqnarray}
where we applied Taylor's expansion of the $\log$ function up to first order and
used notation $q_0= \frac{2q}{2+Tq-q}$.

Using notation
\begin{eqnarray}  \label{D}
D&=& 
\nonumber
     - \frac{T^3}{2 c} \left( \frac{\log(\log(qN) )}{\log(qN)} \right)^2  + 
T^2 \frac{\log(\log(qN))}{cq_0 (\log(q N))^2} + T^3 \frac{\log(\log(qN))}{(c \log(qN))^2} + \\
&+& \left(T \log\left(\frac{q}{p}\right) -\log (T!)\right) 
\left(\frac{T}{c \log(qN)} - T^2 \frac{\log(\log(qN))}{c (\log(qN))^2}  \right) , 
\end{eqnarray} 
\begin{equation} \label{B}
B = T^2 \frac{\log(\log(qN))}{c \log(qN)} -  \frac{T}{c q_0 \log(qN)} +D   
\end{equation} 
and
$$
A=  T\log(\log(qN)) + B,
$$
we have 
$$
m= T \log\left(\frac{q}{p}\right) -\log (T!) + \log(qN) +A + k - \{m(N)\}.
$$
So we obtain that the logarithm of the exponent in \eqref{mu1} is
\begin{eqnarray} \nonumber
L &=& - \frac{T}{c q_0 m} + T \log m -A -k + \{m(N)\} +\OO\left(\frac{1}{(\log N)^2} \right)  
  \nonumber  = \\
&=&     
 - \frac{T}{c q_0 \log (qN) } + T^2 \frac{\log(\log(qN))}{cq_0 (\log(qN))^2} +\nonumber \\
 &+& T \log \big( \log (qN) +T\log(\log(qN)) +B +\log((q/p)^T) - \log(T!) + k - \{ m(N) \} \big) - \nonumber \\
 &-& A - k + \{ m(N) \} +\OO\left(\frac{1}{(\log N)^2 }\right) , \nonumber 
 \end{eqnarray}
 where we used Taylor's expansion for the function $1/x$. 
 Now, by Taylor's expansion for the $\log(x)$ function, we obtain
 \begin{eqnarray}  \nonumber
 L= &-& \frac{T}{c q_0 \log (qN) } + T^2 \frac{\log(\log(qN))}{cq_0 (\log(qN))^2} + T \log(\log (qN)) +\\
 &+& \frac{T \left(T\log(\log(qN)) +B +\log((q/p)^T) - \log(T!) + k - \{ m(N)\} \right) }{c \log (qN)} - \nonumber \\
 &-& \frac{1}{2} \frac{T \left(T\log(\log(qN)) +B +\log((q/p)^T) - \log(T!) + k - \{ m(N)\} \right)^2 }{c (\log(qN))^2} 
 - \nonumber \\
 &-& A - k + \{ m(N) \} +\OO\left(\frac{1}{(\log N)^2 }\right) . \nonumber  
 \end{eqnarray}
 We see that $B$ can be omitted from the quadratic term, and we can apply that $A=  T\log(\log(qN)) + B$,
  so we obtain
 \begin{eqnarray}  \nonumber
 L= &-& \frac{T}{c q_0 \log (qN) } +  \frac{T^2\log(\log(qN))}{cq_0 (\log(qN))^2} +  \frac{T^2 \log(\log(qN))}{c \log(qN)}  \nonumber \\
 &+& \frac{T(\log((q/p)^T) - \log(T!))}{c \log(qN)} +  \frac{T^3\log(\log(qN))}{(c \log(qN))^2} - 
 \frac{T^2}{q_0(c \log(qN))^2} + \frac{T D}{c \log(qN)} + \nonumber \\
 &+& \frac{T(k - \{ m(N)\})}{c \log(qN)} - \frac{1}{2}\frac{T^3(\log(\log(qN)))^2}{c (\log(qN))^2} - \nonumber \\
 &-& \frac{1}{2} \frac{T \left(\log((q/p)^T) - \log(T!) + k - \{ m(N)\} \right)^2 }{c (\log(qN))^2} - \nonumber \\
 &-& \frac{2T}{2} \frac{T\log(\log(qN)) \left(\log((q/p)^T) - \log(T!) + k - \{ m(N)\} \right) }{c (\log(qN))^2} - \nonumber \\
 &-& B - k + \{ m(N) \} +\OO\left(\frac{1}{(\log N)^2 }\right) = \nonumber  \\
 =&&  (k - \{ m(N)\}) \left(\frac{T}{c  \log (qN) } -  \frac{T^2\log(\log(qN))}{c(\log(qN))^2} -1 \right) 
 +\OO\left(\frac{1}{(\log N)^2 }\right) .\nonumber
 \end{eqnarray}
 
So we have by Taylor's expansion
\begin{eqnarray}  \label{e5}
 && e^{-\left(q- \frac{T}{m}\right) N P(A_1)}  =  e^{-p^{-L}} = \\
&=& \exp\left( {- p^{-\left( ( \{ m(N) \} -k) \left(1-\frac{T}{c \log (qN)}
+ \frac{T^2\log(\log(qN))}{c(\log(qN))^2}\right) + 
\OO\left( 1/{(\log N)^2 } \right)\right)} }  \right)=  \nonumber \\
&=& \exp\left({- p^{-\left( ( \{ m(N) \} -k) \left(1-\frac{T}{c \log N} + 
\frac{T^2\log(\log(qN))}{c(\log(qN))^2}  \right)\right)}  
\left(1+\OO\left( {1}/{(\log N)^2 } \right)\right)} \right)  =  \nonumber \\
&=& 
\exp\left({- p^{-\left(( \{ m(N) \} -k) \left(1-\frac{T}{c \log N} + 
\frac{T^2\log(\log(qN))}{c(\log(qN))^2}  \right)\right)} }\right)
\left(1+\OO\left( {1}/{(\log N)^2 } \right)\right) . \nonumber 
\end{eqnarray}
So, from equations \eqref{mu1} and \eqref{e5}, we obtained the desired result.
\end{proof}
\section{Simulation results} \label{SectSimu}
\setcounter{equation}{0}
In this section, first we present simulation results showing the numerical behaviour of 
the first hitting time $\tau_m$. 
These results support Theorem \ref{propT12km}.

Then, we present simulation results for $\mu(N)$, i.e. for the length of the longest $T$ contaminated run.
They show that our new approximation in Theorem \ref{propRun} is better than the former one quoted in 
Proposition \ref{GSW}.
We implemented the simulation in Matlab.
\begin{exmp} \label{TK0}
In this example $p=0.5$, $T=1$.
The length of the coin tossing experiment is $N=10^6$, 
the number of the repetitions of the experiment is $s=2000$.
On parts (a) and (b) of Figure \ref{TK0fig} sign $\oo$ shows the theoretical asymptotic probability and
$\ast$ shows the relative frequency of those experiments when $\mu(N)$, 
that is the longest $T$-contaminated run is shorter than the given value on the horizontal axis.
Part (a) of Figure \ref{TK0fig} shows the fit of the empirical distribution of $\mu(N)$ to the asymptotic distribution
 given by our Theorem \ref{propRun}. The fit is good, the Kolmogorov distance is $0.0264$.
Part (b) of Figure \ref{TK0fig} shows the fit of the empirical distribution of $\mu(N)$ to the asymptotic distribution
 given by the old result quoted in Proposition \ref{GSW}. The fit is poor, the Kolmogorov distance is $0.0778$.
Part (c) of Figure \ref{TK0fig} shows the first hitting time of the $T$-contaminated run having length $n=13$.
The solid line is the empirical distribution given by the simulation, the dashed line is the asymptotic theoretical
 distribution presented in Theorem \ref{propT12km}. The fit is good.
\begin{figure}[h!]
    \centering
    \begin{subfigure}[b]{0.30\textwidth}
        \includegraphics[width=\textwidth]{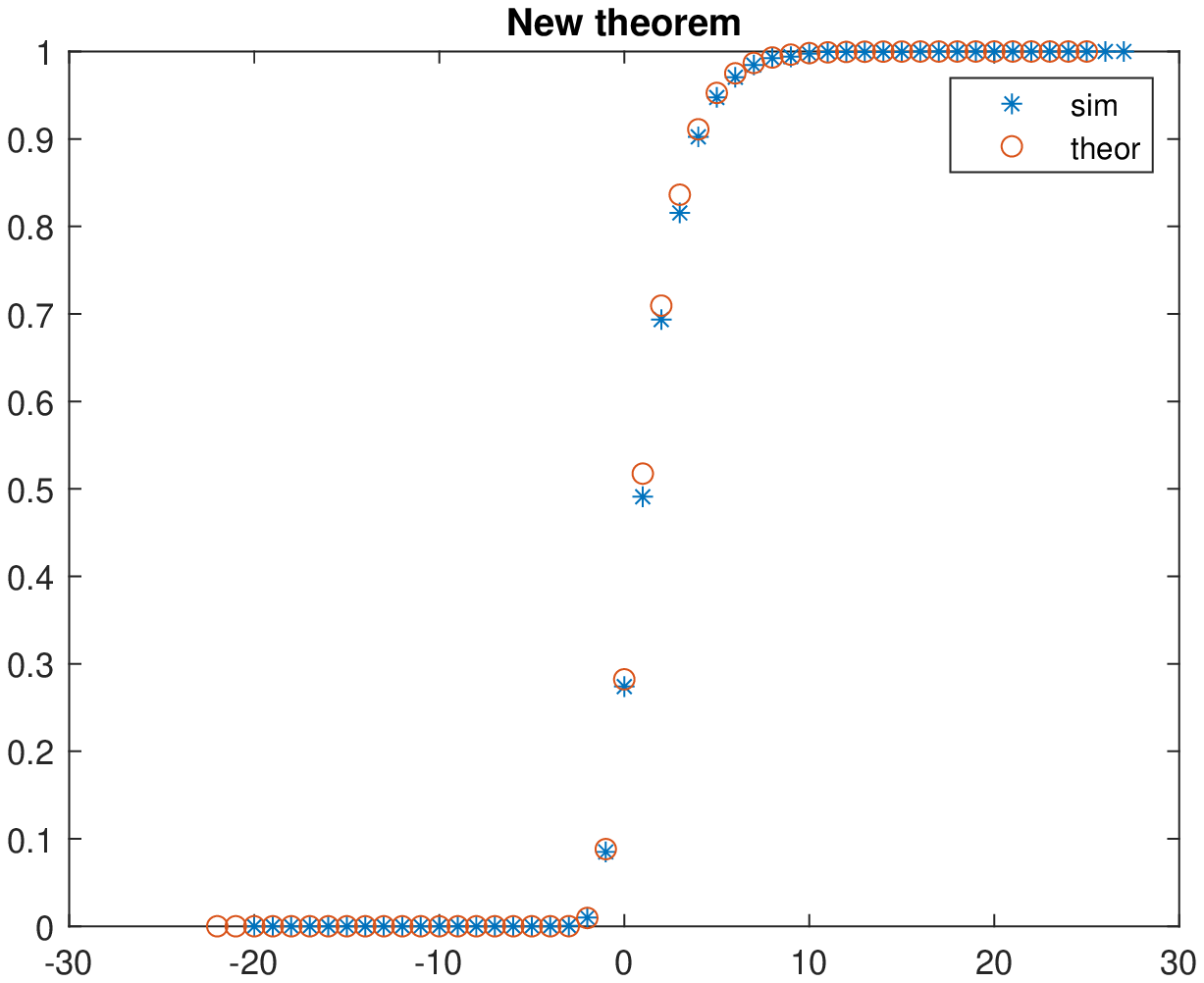}
        \caption*{(a) Longest run (new) }
    \end{subfigure}
    \begin{subfigure}[b]{0.30\textwidth}
        \includegraphics[width=\textwidth]{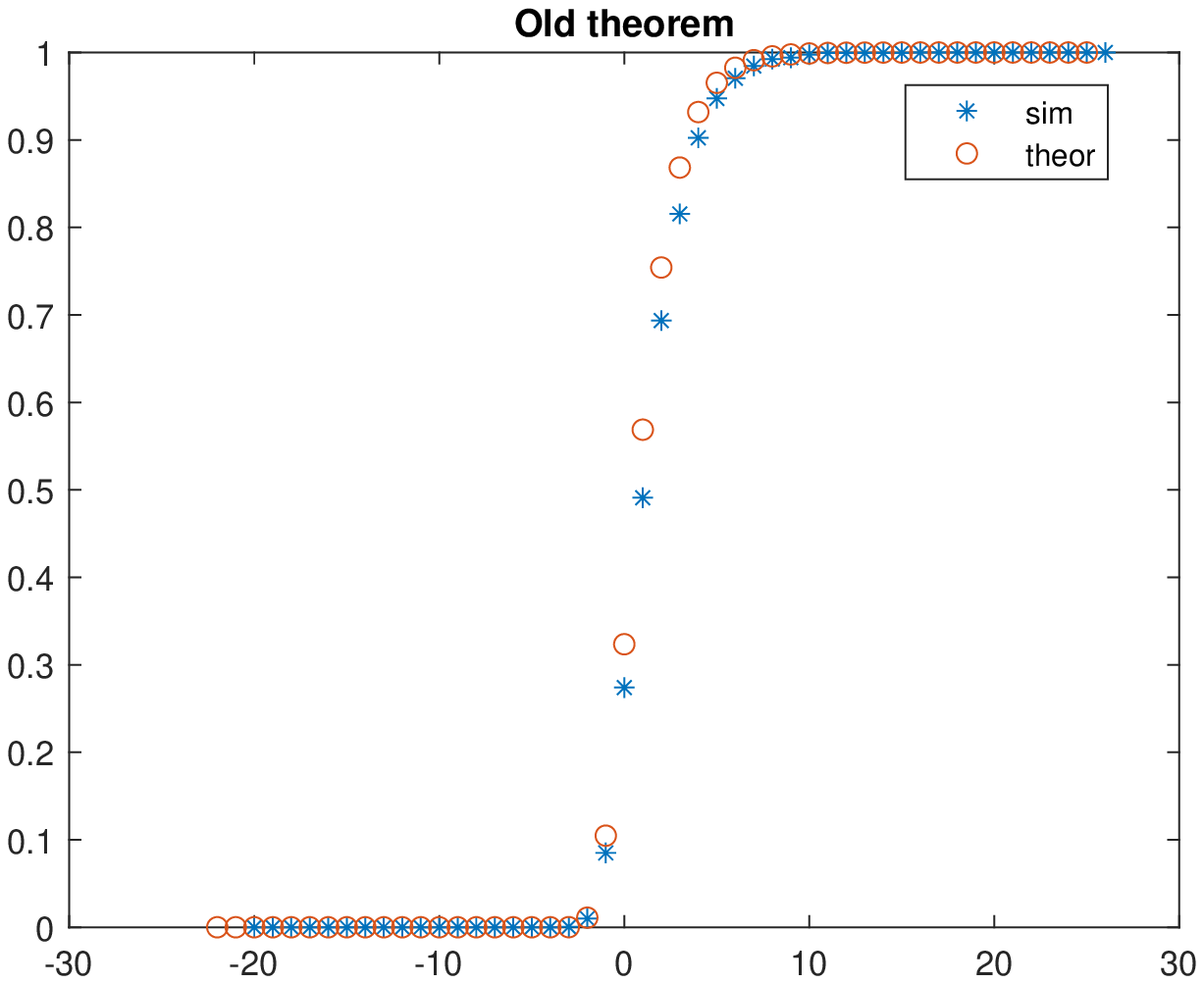}
        \caption*{(b) Longest run (old)}
    \end{subfigure}
    \begin{subfigure}[b]{0.30\textwidth}
        \includegraphics[width=\textwidth]{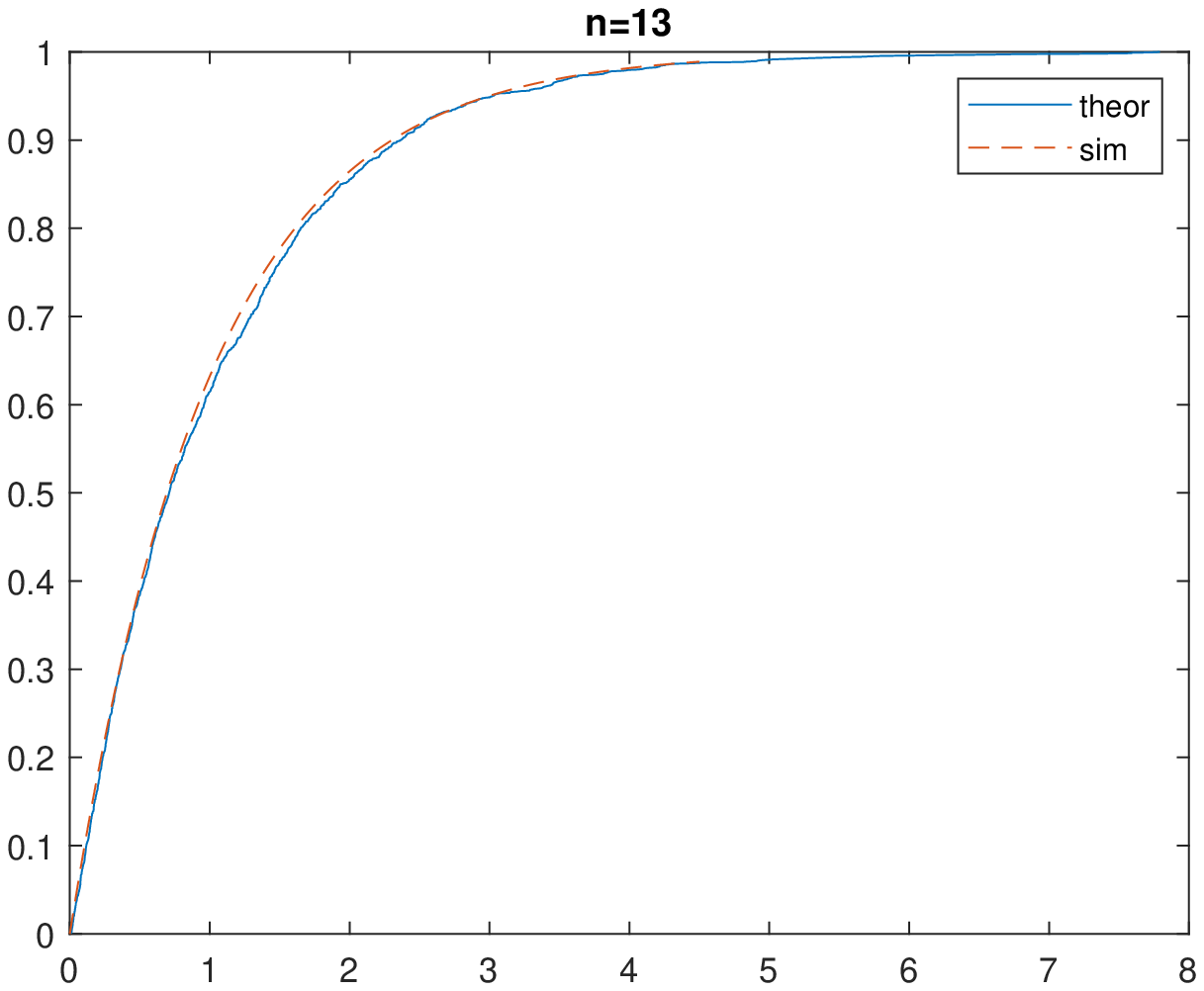}
         \caption*{(c) First hitting time }
    \end{subfigure}
    \caption{$T=1$, $p=0.5$, $N= 10^6$, $s=2000$}\label{TK0fig}
 \end{figure}   
\end{exmp}
\begin{exmp} \label{TK1}
In this example $p=0.5$, $T=2$.
The length of the coin tossing experiment is $N=10^6$, 
the number of the repetitions of the experiment is $s=2000$.
On parts (a) and (b) of Figure \ref{TK1fig} sign $\oo$ shows the theoretical asymptotic probability and
$\ast$ shows the relative frequency of those experiments when $\mu(N)$, 
that is the longest $T$-contaminated run is shorter than the given value on the horizontal axis.
Part (a) of Figure \ref{TK1fig} shows the fit of the empirical distribution of $\mu(N)$ to the asymptotic distribution
 given by our Theorem \ref{propRun}. The fit is good, the Kolmogorov distance is $0.0148$.
Part (b) of Figure \ref{TK1fig} shows the fit of the empirical distribution of $\mu(N)$ to the asymptotic distribution
 given by the old result quoted in Proposition \ref{GSW}. The fit is poor, the Kolmogorov distance is $0.2129$.
Part (c) of Figure \ref{TK1fig} shows the first hitting time of the $T$-contaminated run having length $n=27$.
The solid line is the empirical distribution given by the simulation, the dashed line is the asymptotic theoretical
 distribution presented in Theorem \ref{propT12km}. The fit is good.
\begin{figure}[h!]
    \centering
    \begin{subfigure}[b]{0.30\textwidth}
        \includegraphics[width=\textwidth]{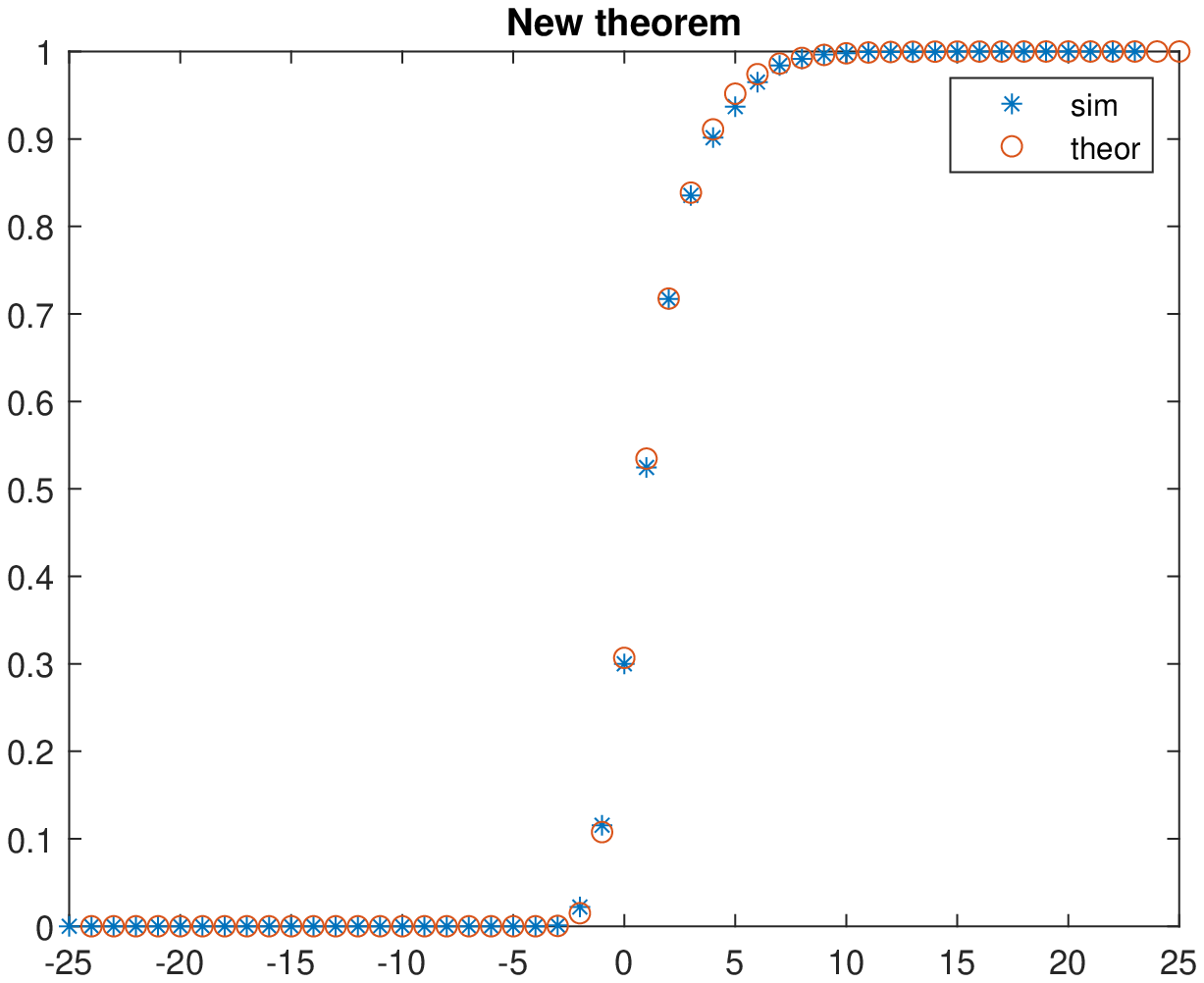}
        \caption*{(a) Longest run (new) }
    \end{subfigure}
    \begin{subfigure}[b]{0.30\textwidth}
        \includegraphics[width=\textwidth]{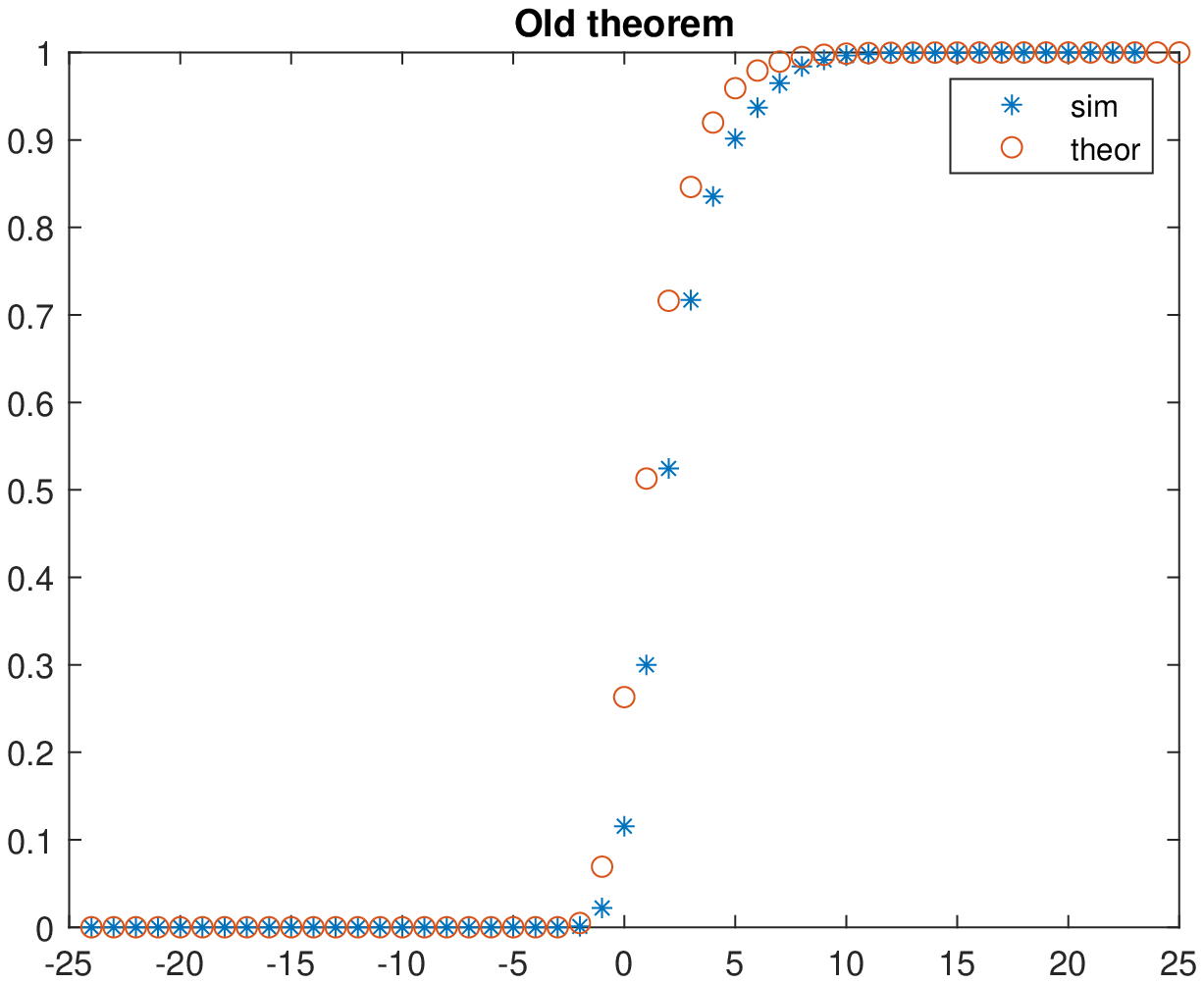}
        \caption*{(b) Longest run (old)}
    \end{subfigure}
    \begin{subfigure}[b]{0.30\textwidth}
        \includegraphics[width=\textwidth]{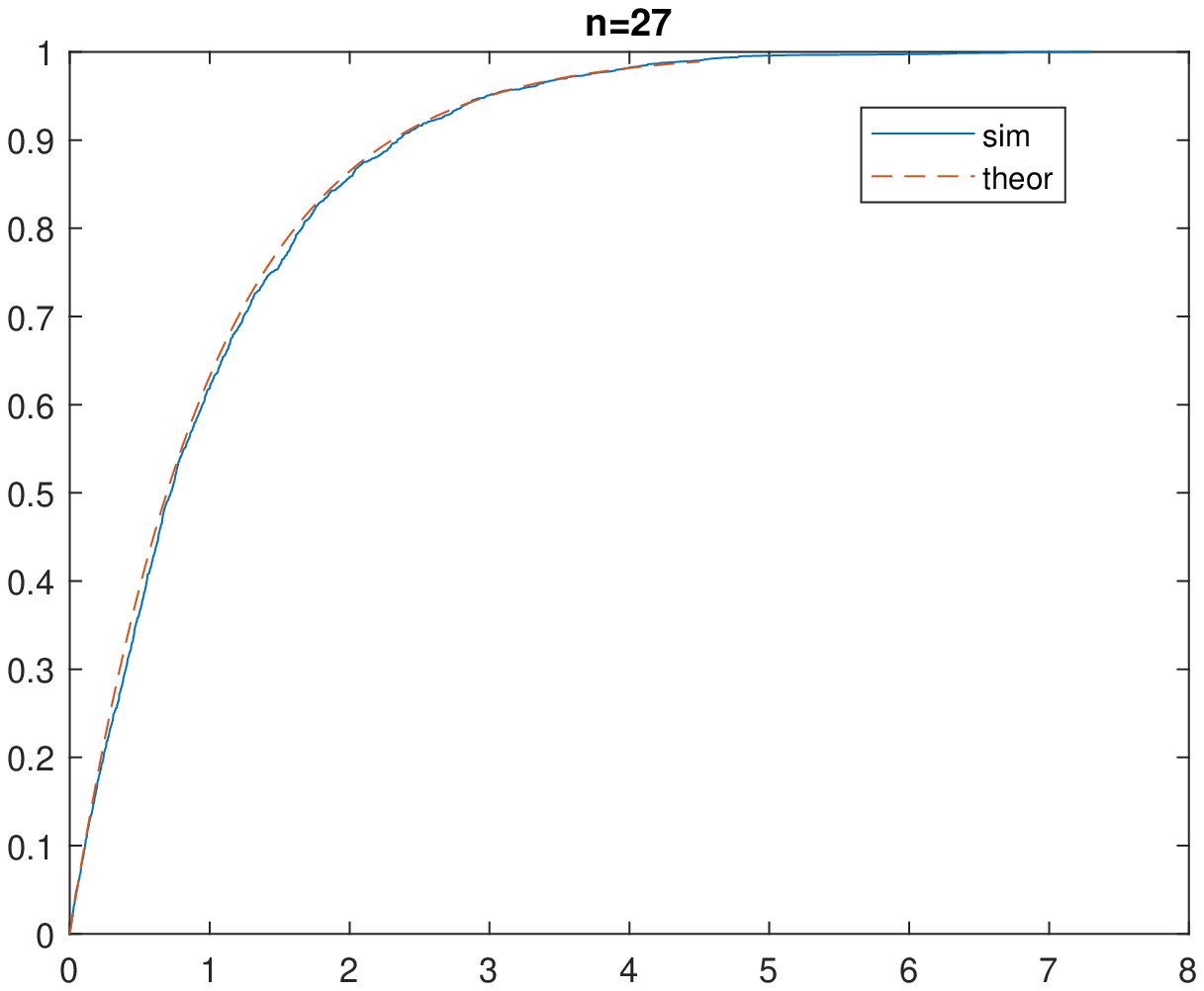}
         \caption*{(c) First hitting time }
    \end{subfigure}
    \caption{$T=2$, $p=0.5$, $N= 10^6$, $s=2000$}\label{TK1fig}
 \end{figure}   
\end{exmp}
\begin{exmp} \label{TK2}
In this example $p=0.6$, $T=2$.
The length of the coin tossing experiment is $N=10^6$, 
the number of the repetitions of the experiment is $s=2000$.
Part (a) of Figure \ref{TK2fig} shows the fit of the empirical distribution of $\mu(N)$ to the asymptotic distribution
 given by our Theorem \ref{propRun}. The fit is good, the Kolmogorov distance is $0.0250$.
Part (b) of Figure \ref{TK2fig} shows the fit of the empirical distribution of $\mu(N)$ to the asymptotic distribution
 given by the old result quoted in Proposition \ref{GSW}. The fit is poor, the Kolmogorov distance is $0.1953$.
Part (c) of Figure \ref{TK3fig} shows the first hitting time of the $T$-contaminated run having length $n=27$.
The solid line is the empirical distribution given by the simulation, the dashed line is the asymptotic theoretical
 distribution presented in Theorem \ref{propT12km}. The fit is good.
\begin{figure}[h!]
    \centering
    \begin{subfigure}[b]{0.30\textwidth}
        \includegraphics[width=\textwidth]{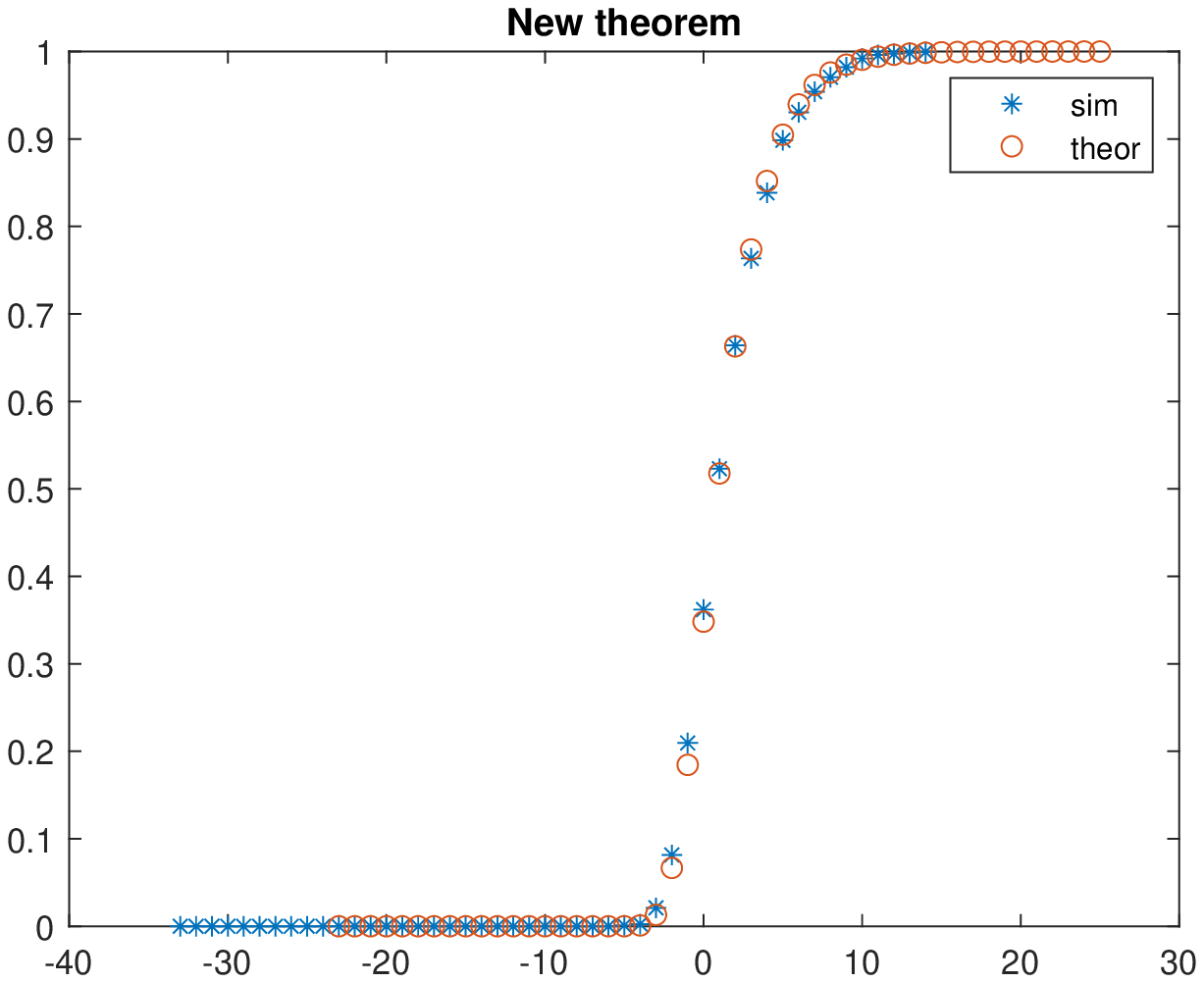}
        \caption*{(a) Longest run (new) }
    \end{subfigure}
    \begin{subfigure}[b]{0.30\textwidth}
        \includegraphics[width=\textwidth]{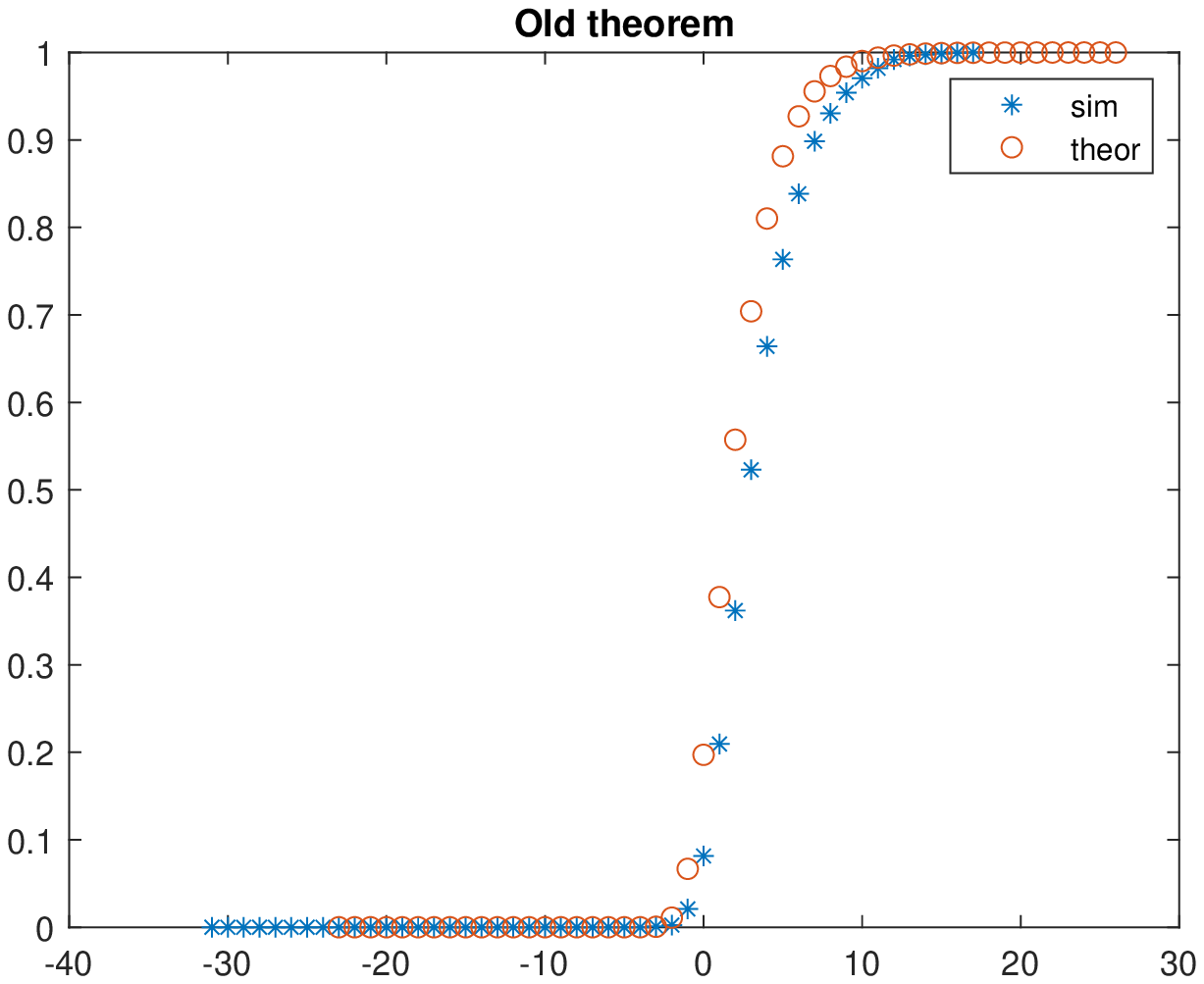}
        \caption*{(b) Longest run (old)}
    \end{subfigure}
    \begin{subfigure}[b]{0.30\textwidth}
        \includegraphics[width=\textwidth]{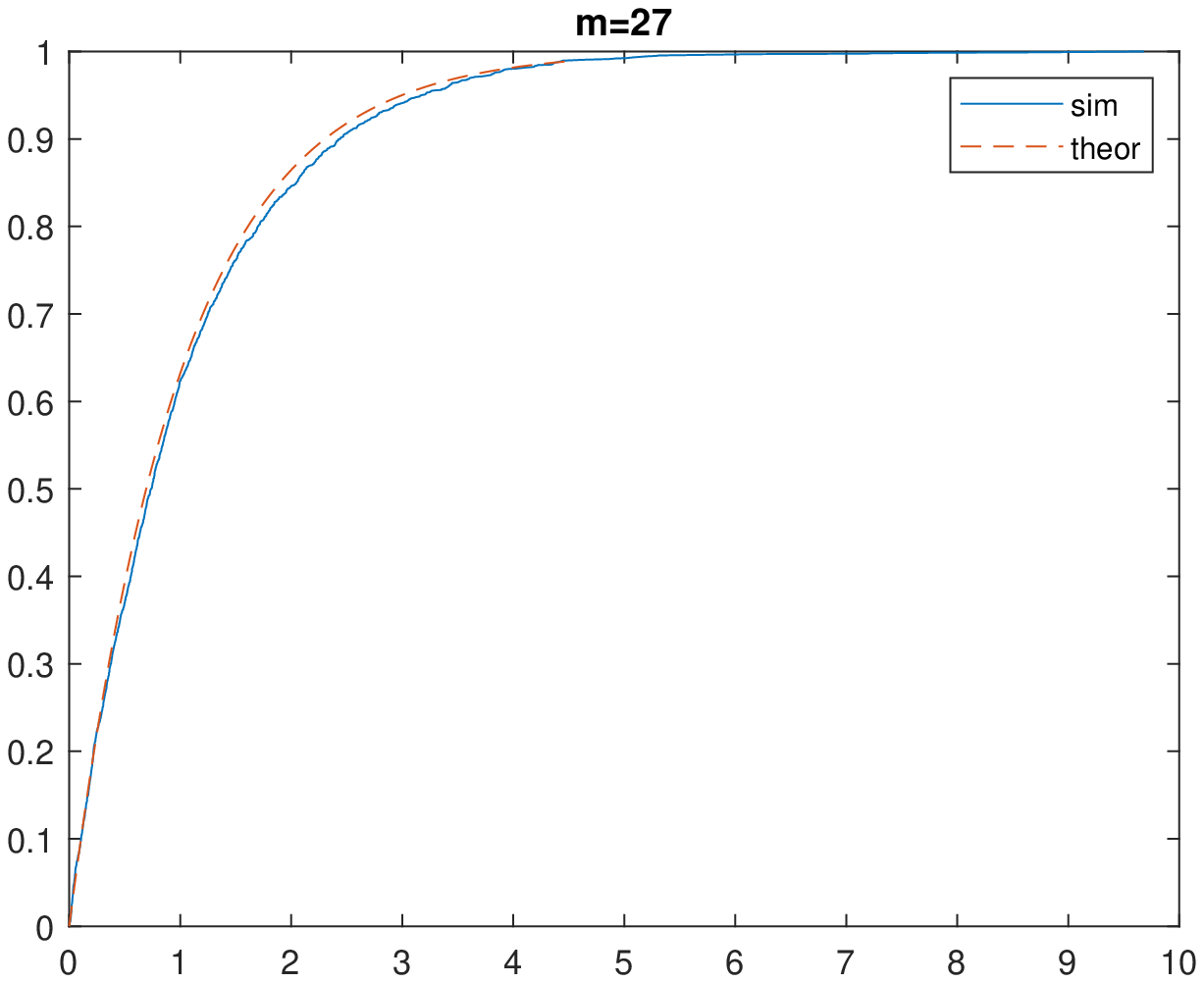}
         \caption*{(c) First hitting time }
    \end{subfigure}
    \caption{$T=2$, $p=0.6$, $N= 10^6$, $s=2000$}\label{TK2fig}
 \end{figure}   
\end{exmp}
\begin{exmp} \label{TK3}
In this example $p=0.4$, $T=2$.
The length of the coin tossing experiment is $N=10^6$, 
the number of the repetitions of the experiment is $s=500$.
Part (a) of Figure \ref{TK3fig} shows the fit of the empirical distribution of $\mu(N)$ to the asymptotic distribution
 given by our Theorem \ref{propRun}. The fit is good, the Kolmogorov distance is $0.0172$.
Part (b) of Figure \ref{TK3fig} shows the fit of the empirical distribution of $\mu(N)$ to the asymptotic distribution
 given by the old result quoted in Proposition \ref{GSW}. The fit is poor, the Kolmogorov distance is $0.1948$.
Part (c) of Figure \ref{TK3fig} shows the first hitting time of the $T$-contaminated run having length $n=17$.
The solid line is the empirical distribution given by the simulation, the dashed line is the asymptotic theoretical
 distribution presented in Theorem \ref{propT12km}. The fit is good.
\begin{figure}[h!]
    \centering
    \begin{subfigure}[b]{0.30\textwidth}
        \includegraphics[width=\textwidth]{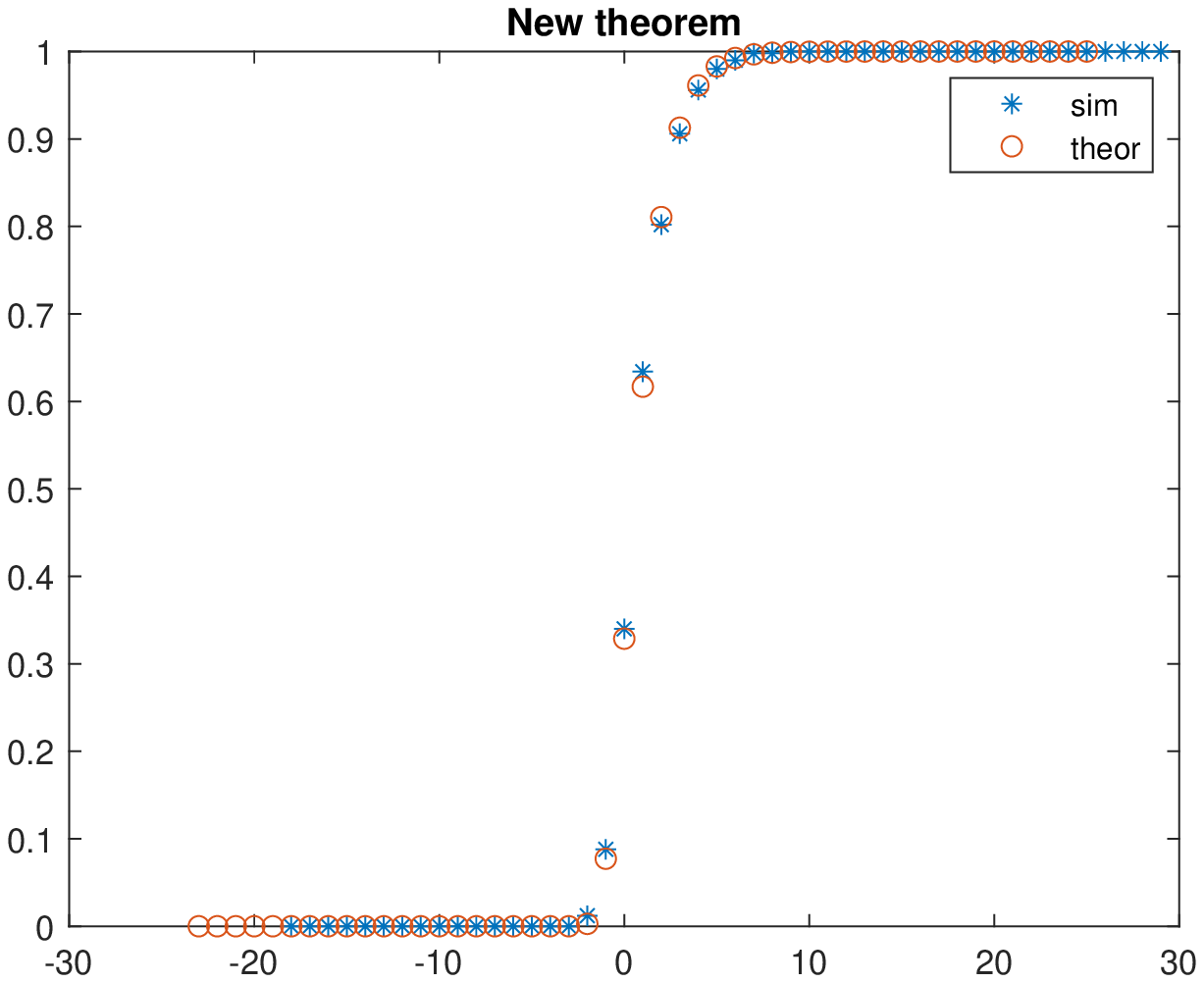}
        \caption*{(a) Longest run (new) }
    \end{subfigure}
    \begin{subfigure}[b]{0.30\textwidth}
        \includegraphics[width=\textwidth]{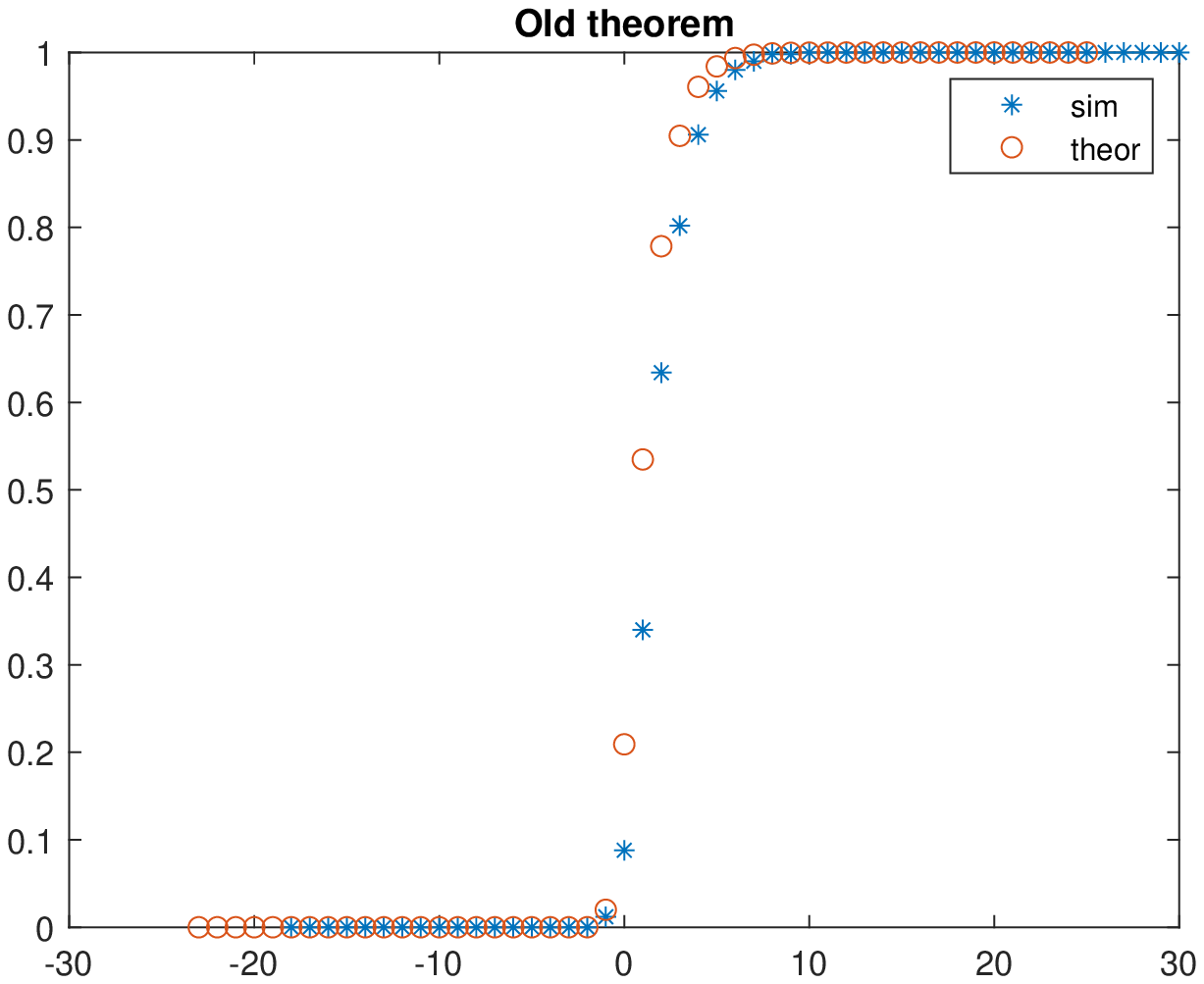}
        \caption*{(b) Longest run (old)}
    \end{subfigure}
    \begin{subfigure}[b]{0.30\textwidth}
        \includegraphics[width=\textwidth]{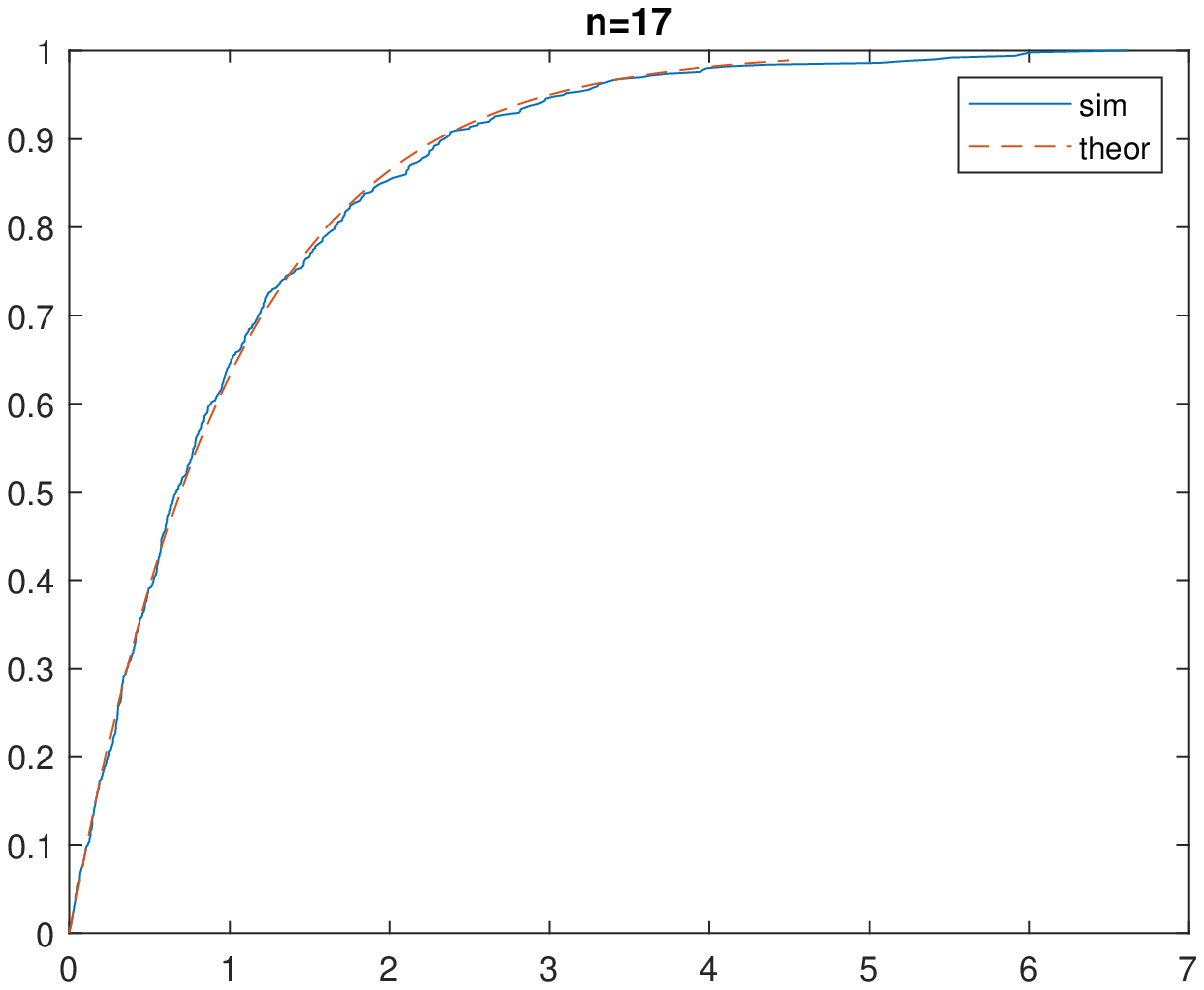}
         \caption*{(c) First hitting time }
    \end{subfigure}
    \caption{$T=2$, $p=0.4$, $N= 10^6$, $s=500$}\label{TK3fig}
 \end{figure}   
\end{exmp}

\end{document}